\documentclass[leqno,bsl,bibother]{asl}

\newtheorem{proposition}{Proposition}[section]
\newtheorem{lemma}[proposition]{Lemma}
\newtheorem{corollary}[proposition]{Corollary}

\theoremstyle{definition}
\newtheorem{definition}[proposition]{Definition}

\newtheorem{remark}[proposition]{Remark}

\renewcommand{\labelenumi}{{\normalfont (\roman{enumi})}}

\newif\ifautoqed
\autoqedtrue

\newif\ifasl
\newif\ifautoproof

\def\mathtext#1#2{\ifmmode#1{#2}\else$#1{#2}$\fi}
\def\mathdefzero#1#2#3{\relax\def#1{\relax\mathtext{#3}{#2}}}
\def\mathdefone#1#2#3{\relax\def#1##1{\relax\def\1{{##1}}\mathtext{#3}{#2}}}
\def\mathdeftwo#1#2#3{\relax\def#1##1##2{\relax\def\1{{##1}}
   \def\2{{##2}}\mathtext{#3}{#2}}}
\def\mathdefthree#1#2#3{\relax\def#1##1##2##3{\relax\def\1{{##1}}
   \def\2{{##2}}\def\3{{##3}}\mathtext{#3}{#2}}}
\def\mmdef#1#2#3#4{\ifcase#1\mathdefzero{#2}{#3}{#4}
   \or\mathdefone{#2}{#3}{#4}   \or\mathdeftwo{#2}{#3}{#4}
   \or\mathdefthree{#2}{#3}{#4} \or\mathdeffour{#2}{#3}{#4}\fi}

\def\mdef#1#2[#3]{\mmdef{#1}{#2}{#3}{}}
\def\reldef#1#2[#3]{\mmdef{#1}{#2}{#3}{\mathrel}}
\def\bindef#1#2[#3]{\mmdef{#1}{#2}{#3}{\mathbin}}
\def\opdef#1#2[#3]{\mmdef{#1}{#2}{#3}{\mathop}}

\def\ssfdef#1[#2]{\def#1{{\ssf {#2}}}}
\mdef3\lis[\1_{\2}\ldots\1_{\3}]
\mdef3\llist[\1_{\2},\ldots,\1_{\3}]
\mdef2\zlist[\1_0,\ldots,\1_{\2-1}]
\mdef2\zlis[\1_0\ldots\1_{\2-1}]
\mdef3\andlist[\1_{\2}\land\cdots\land\1_{\3}] 
\mdef3\orlist[\1_{\2}\lor\cdots\lor\1_{\3}]

\mdef1\:[\dot{\overbrace{\1}}] 
\mdef1\.[\dot{\1}]
\mdef1\modelof[\>\models{\kern-9pt}_%
   {\lower5pt\hbox{$\scriptstyle\1$}}\>]
\def\notmodels{\mathrel|\joinrel\not=}
\mdef1\notmodelof[\>\notmodels{\kern-9pt}_%
   {\lower5pt\hbox{$\scriptstyle\1$}}\>]

\def\incl{\subseteq}
\mdef1\set[\left\{\mkern1.5mu\1\mkern1.5mu\right\}]           
\mdef1\bigset[\bigl\{\mkern1.5mu\1\mkern1.5mu\bigr\}]           
\mdef3\setlist[\left\{\mkern2mu{\1}_{\2},\ldots,
     {\1}_{\3}\mkern2mu\right\}]
\mdef2\setzlist[\left\{\mkern2mu{\1}_0,\ldots,
     {\1}_{\2-1}\mkern2mu\right\}]
\mdef2\setof[\left\{\,\1:\2\,\right\}]                     
\mdef2\bigsetof[\bigl\{\,\1:\2\,\bigr\}]                     
\mdef2\Bigsetof[\Bigl\{\,\1:\2\,\Bigr\}]                     
\def\power{\wp}

\def\Iff{\,\Longleftrightarrow\,}

\def\implies{\,\rightarrow\,}
\def\Implies{\,\Longrightarrow\,}

\def\qand{\quad\hbox{and}\quad}

\def\qIff{\quad\Iff\quad}

\def\qImplies{\quad\Implies\quad}

\def\DDegree#1#2{{\ssbf#1}\ifcat#20
   \ifcase#2\or'\or''\or'''\else^{(#2)}\fi\else
   ^{(#2)}\fi} 
\def\degree#1#2{\ifmmode\DDegree{#1}{#2}\else
   $\DDegree{#1}{#2}$\fi} 
\def\lnot{\neg\>}
\mdef1\alef[\aleph_{\1}]

\mdef2\prf[\{\1\}^{\2}]
\mdef3\pprf[\{\1\}^{\2}_{\3}]

\def\lbrac{[}
\def\rbrac{]}
\mdef3\mm[{\ssf m}^{\1}_{\2,\3}]

\mdef1\tree[{\ssf T}_{\1}]
\mdef2\treelg[{\ssf T}_{\1}^{\2}]
\mdef2\ttree[{\ssf T}_{\1,\2}]
\mdef2\ttreelg[{\ssf T}_{\1,\2}^{\2}]
\mdef3\ttreelgg[{\ssf T}_{\1,\2}^{\3}]

\mdef1\trst[{\ssf V}^{\1}]
\mdef2\ptrst[{\ssf V}^{\1}_{\2}]
\mdef2\pltrst[{\ssf V}^{\1}_{<\2}]

\ssfdef\ext[Ext]

\mdef1\gs[{\ssf S}_{\1}]
\mdef1\gsf[{\ssf S}^F_{\1}]
\mdef2\gsd[{\cal S}^{\1}_{\2}]
\mdef3\gsdd[{\cal S}^{\1,\,\2}_{\3}]
\mdef2\gsdb[\overline{\cal S}^{\1}_{\2}]
\mdef0\deg[{\ssf dg}]
\mdef0\degs[{\ssf dg}_{\ssf s}]
\mdef0\degt[{\ssf dg}_T]
\mdef0\degw[{\ssf dg}_{\ssf w}]
\mdef0\Dgt[{\mathbb D}_T]
\mdef0\Dgs[{\mathbb D}_{\ssf s}]
\mdef0\Dgw[{\mathbb D}_{\ssf w}]
\mdef0\Dgsdual[{\mathbb D}_{\ssf s}^\dualmarker]
\mdef0\Dgwdual[{\mathbb D}_{\ssf w}^\dualmarker]
\mdef0\Dgpt[{\mathbb P}_T]
\mdef0\Dgps[{\mathbb P}_{\ssf s}]
\mdef0\Dgpw[{\mathbb P}_{\ssf w}]
\reldef0\leqs[\leq_{\ssf s}]
\reldef0\leqw[\leq_{\ssf w}]
\reldef0\geqw[\geq_{\ssf w}]
\reldef0\les[<_{\ssf s}]
\reldef0\lew[<_{\ssf w}]
\mdef1\om[{}^\omega\1]
\mdef1\dnr[{\ssf DNR}_{\1}]

\mdef3\zslist[\1_0^{\3},\ldots,\1_{\2-1}^{\3}]
\mdef2\pre[{}^{\1}\2]

\def\mm{{\ssf m}}
\def\Impliess{\;\Implies\;}

\catcode`\@=11
\def\cases#1{\left\{\,\vcenter
   {\m@th\ialign{$##\hfil$&\quad##\hfil\crcr#1\crcr}}\right.}
\catcode`\@=12

\def\join{\mathchoice{\mathrel{\rlap{$\vee$}\mskip1.8 mu\vee}}{\mathrel{\rlap{$\vee$}\mskip1.8 mu\vee}}{\mathrel{\rlap{$\scriptstyle\vee$}\mskip1.8 mu\vee}}{\mathrel{\rlap{$\scriptscriptstyle\vee$}\mskip1.8 mu\vee}}}
\def\meet{\mathchoice{\mathrel{\rlap{$\wedge$}\mskip1.8 mu\wedge}}{\mathrel{\rlap{$\wedge$}\mskip1.8 mu\wedge}}{\mathrel{\rlap{$\scriptstyle\wedge$}\mskip1.8 mu\wedge}}{\mathrel{\rlap{$\scriptscriptstyle\wedge$}\mskip1.8 mu\wedge}}}
\opdef0\latneg[\rlap{$\lnot$}\mskip2mu\lnot\kern-2pt]
\opdef0\bigjoin[\rlap{$\displaystyle\bigvee$}\displaystyle\bigvee]
\opdef0\bigmeet[\rlap{$\displaystyle\bigwedge$}\displaystyle\bigwedge]
\bindef0\latimpl[\rlap{\kern.4pt$\rightarrow$}\rightarrow]

\def\djoin{\buildrel\lower3pt\hbox{$\scriptscriptstyle\dualmarker$}\over\join}
\def\dmeet{\buildrel\lower3pt\hbox{$\scriptscriptstyle\dualmarker$}\over\meet}
\reldef0\dimbeds[\buildrel\lower5pt\hbox{$\scriptscriptstyle\dualmarker$}\over\hookrightarrow]
\let\dembeds=\dimbeds

\reldef0\dlatneg[\,\buildrel\lower2pt\hbox{$\scriptscriptstyle\dualmarker$}\kern1pt\over\latneg]

\def\dlatimpl{\buildrel\lower2pt\hbox{$\scriptscriptstyle\dualmarker$}\over\latimpl}
\def\dleq{\buildrel\lower2pt\hbox{$\scriptscriptstyle\dualmarker$}\over\leq}
\let\dualmarker=\circ

\def\latz{{\bf 0}}
\def\lato{{\bf 1}}

\def\dgp{{\bf p}}
\def\dgq{{\bf q}}
\def\dgr{{\bf r}}

\mdef0\calp[{\cal P}]
\mdef0\calq[{\cal Q}]
\mdef0\calr[{\cal R}]
\mdef0\cals[{\cal S}]
\mdef0\calt[{\cal T}]
\mdef0\calu[{\cal U}]
\mdef0\calv[{\cal V}]
\mdef0\calw[{\cal W}]
\mdef0\calx[{\cal X}]
\mdef0\caly[{\cal Y}]
\mdef0\calz[{\cal Z}]
\mdef0\truth[{\ssf T}]
\mdef0\falsity[{\ssf F}]

\def\Ifff{\,\Iff\,}
\def\dIff{\;:\Longleftrightarrow\;}

\def\dff#1{{\it#1}}
\def\dff#1{{\bfit #1}}
\let\ssf=\sf

\def\bframe#1{\begin{frame}\frametitle{#1}}

\let\iso=\simeq
\reldef0\submeet[\rlap{$\scriptscriptstyle\wedge$}\mskip1.5 mu{\scriptscriptstyle\wedge}]
\reldef0\subjoin[\rlap{$\scriptscriptstyle\vee$}\mskip1.5 mu{\scriptscriptstyle\vee}]
\mdef0\freedist[\mathbb{F}_{\submeet}]
\mdef0\free[\mathbb{F}]
\bindef0\sublatimpl[\rlap{\kern.5pt$\scriptscriptstyle\rightarrow$}{\scriptscriptstyle\rightarrow}]
\def\subdlatimpl{\,\buildrel\lower2pt\hbox{$\scriptscriptstyle\dualmarker$}\over\sublatimpl\,}

\def\Th{{\ssf Th}}
\def\Thd{{\ssf Th}^\dualmarker}

\def\IPC{{\ssf IPC}}
\def\WEM{{\ssf WEM}}

\reldef0\pripc[\vdash_{\kern-4pt\lower.5pt\hbox{$\scriptscriptstyle\IPC$}}]
\reldef0\prx[\vdash_{\kern-4pt\lower.5pt\hbox{$\scriptscriptstyle X$}}]
\reldef0\prwem[\vdash_{\kern-4pt\lower.5pt\hbox{$\scriptscriptstyle\WEM$}}]
\reldef0\notpripc[\not\vdash_{\kern-2pt\lower.5pt\hbox{$\scriptscriptstyle\IPC$}}]
\reldef0\notprwem[\not\vdash_{\kern-2pt\lower.5pt\hbox{$\scriptscriptstyle\WEM$}}]

\let\mathstr=\mathfrak
\def\forces{\mathrel{\raise 7pt
   \hbox{$\mathchar "330C\mkern-3.5mu \mathchar  "330C$}
   \mkern-7.5mu
   {\kern 1.5pt\hbox{$\raise .5 pt\hbox{}
   \over\kern6pt$}}}}
\def\forces{\Vdash}
\mdef1\random[{\ssf R}_#1]

\newif\ifendpf
\def\squarebox#1{\vbox{\hrule\hbox{\vrule height#1\kern#1\vrule}\hrule}}
\def\endpf{\ifendpf\ifmmode\enspace\fi\squarebox{6pt}\fi\global\endpffalse}
\def\noproof{\ \qed}
\def\p#1{{\ssf p}_{#1}}

\mdef0\l[{\mathstr L}]
\mdef0\lz[{\mathstr L}_{\ssf 0}]
\mdef0\lzo[{\mathstr L}_{\ssf 0}^{\ssf 1}]
\mdef0\lo[{\mathstr L}^{\ssf 1}]

\let\cal=\mathcal
\asltrue
\autoprooffalse

\title{Addendum to: A survey of Mu\v cnik and Medvedev degrees}

\author{Peter G. Hinman}
\address{Department of Mathematics\\
University of Michigan\\
2074 East Hall\\
530 Church Street\\
Ann Arbor, MI 48109-1043, USA}
\email{pgh@umich.edu}

\date{}

\begin{document}

\begin{abstract}We include here some material that did not make its way into the published version [2], in particular a proof of Theorem K to the effect that there is an initial segment of the strong degrees with dual theory \IPC, the intuitionistic propositional calculus.
\end{abstract}

\maketitle 

\tableofcontents

\setcounter{section}{16}

\begin{section}{Implicative lattices (2)}
 This section presents some additional facts about (dual-)implicative lattices in general and \Dgs\ in particular. 
 
 Lemma \ref{embedtheory} showed that when ${\mathstr L}\embeds{\mathstr K}$ is a lattice embedding that respects $\latimpl$, then $\Th({\mathstr K})\incl\Th({\mathstr L})$. We write ${\mathstr K}\rightarrowtail{\mathstr L}$ iff there exists a surjection of $\mathstr K$ onto $\mathstr L$ that respects the lattice operations and $\latimpl$ and have similarly
 
 \begin{lemma}
 For any implicative lattices $\mathstr K$ and $\mathstr L$, if ${\mathstr K}\rightarrowtail{\mathstr L}$, then $\Th({\mathstr K})\incl\Th({\mathstr L})$.
 \end{lemma}
 
 \begin{proof}
 If $\eta$ is a surjection of $\mathstr K$ onto $\mathstr L$ as described, then for any $\mathstr L$-valuation $w$, there exists a function $v_0$ such that for each atomic sentence {\ssf p}, $w({\ssf p})=\eta(v_0({\ssf p}))$ and hence a $\mathstr K$-valuation $v$ such that $w=\eta\circ v$. Then for any $\phi\in\Th({\mathstr K})$ and any $\mathstr L$-valuation $w$,
 \[w(\phi)=\eta(v(\phi))=\eta(\lato_{\mathstr K})=\lato_{\mathstr L},\]
 so $\phi\in\Th({\mathstr L})$.
 \end{proof}
 
 We shall need a particular family of instances of this (see Definition \ref{segments}):
 
 \begin{lemma}
 For any implicative lattice $\mathstr L$ and $c,d,e\in L$, ${\mathstr L}[c,d]\rightarrowtail{\mathstr L}[c\meet e, d\meet e]$.
 \end{lemma}
 
 \begin{proof}
 Given $\mathstr L$ and $c$, $d$, and $e$, set $\eta(a):=a\meet e$. $\eta$ is surjective because for any $f$ with $c\meet e\leq f\leq d\meet e$, $f=\eta(c\join f)$. $\eta$ respects \latz, \lato, $\meet$ and $\join$ by definition and distributivity. For $\eta$ to respect $\latimpl$ (see Lemma \ref{seglatimpl}), means that for any $a,b\in L[c,d]$,
 \[(a\latimpl_d b)\meet e=(a\meet e)\latimpl_{d\meet e}(b\meet e),\]
 that is,
 \[(a\latimpl b)\meet d\meet e=(a\meet e\latimpl b\meet e)\meet d\meet e.\]
 This follows because by the definition of $\latimpl$, $(a\latimpl b)\leq (a\meet e\latimpl b\meet e)$, since
 \[(a\meet e)\meet (a\latimpl b)=a\meet(a\latimpl b)\meet e\leq b\meet e,\]
 and $(a\meet e\latimpl b\meet e)\meet e\leq a\latimpl b$ because 
 \begin{endproofeqnarray*}
 a\meet(a\meet e\latimpl b\meet e)\meet e&=&(a\meet e)\meet(a\meet e\latimpl b\meet e)\\
 &\leq& b\meet e\leq b.
 \end{endproofeqnarray*}
 \end{proof}
 
 \begin{corollary}\label{surjectiveTh}
 For any lattice $\mathstr L$ and $c,d,e\in L$,
 \begin{enumerate}
 \item if $\mathstr L$ is implicative, then $\Th({\mathstr L}[c,d])\incl\Th({\mathstr L}[c\meet e,d\meet e])$;
 \item if $\mathstr L$ is dual-implicative, then $\Thd({\mathstr L}[c,d])\incl\Thd({\mathstr L}[c\join e,d\join e])$.
 \end{enumerate}
 \end{corollary}
 
 \begin{proof}
 Part (i) is immediate from the preceding two lemmas and (ii) is its dual -- explicitly,
 \begin{endproofeqnarray*}
 \Thd({\mathstr L}[c,d])&=&\Th({\mathstr L}^\dualmarker[d,c])\incl\Th({\mathstr L}^\dualmarker[d\dmeet e,c\dmeet e])\\
 &=&\Th({\mathstr L}[c\join e,d\join e]^\dualmarker)=\Thd({\mathstr L}[c\join e,d\join e]).
 \end{endproofeqnarray*}
 \end{proof}
 
 Next we consider some questions of distributivity; we will need only a few of the clauses of the next lemma but include them all for reference.
 
 \begin{lemma}\label{distprops}
 For any lattice $\mathstr L$ and $a$, $b$, $c$, and $d\in L$,
 \begin{enumerate}
\item if $\mathstr L$ is implicative, then
\vskip3pt
\begin{enumerate}
\item[(1)]$a\latimpl(c\meet d)=(a\latimpl c)\meet(a\latimpl d)$
\vskip3pt
\item[(2)]$a\latimpl(c\join d)\geq(a\latimpl c)\join(a\latimpl d)$
\vskip3pt
\item[(3)]$(a\join b)\latimpl c=(a\latimpl c)\meet(b\latimpl c)$
\vskip3pt
\item[(4)]$(a\meet b)\latimpl c\geq(a\latimpl c)\join(b\latimpl c);$
\end{enumerate}
\vskip5pt
\item if $\mathstr L$ is dual-implicative, then
\vskip 3pt
\begin{enumerate}
\item[$(1)^\dualmarker$]$a\dlatimpl(c\meet d)\leq(a\dlatimpl c)\meet(a\dlatimpl d)$
\vskip3pt
\item[$(2)^\dualmarker$]$a\dlatimpl(c\join d)=(a\dlatimpl c)\join(a\dlatimpl d)$
\vskip3pt
\item[$(3)^\dualmarker$]$(a\join b)\dlatimpl c\leq(a\dlatimpl c)\meet(b\dlatimpl c)$
\vskip3pt
\item[$(4)^\dualmarker$]$(a\meet b)\dlatimpl c=(a\dlatimpl c)\join(b\dlatimpl c).$
\end{enumerate}
\end{enumerate}
\end{lemma}

\begin{proof}
These are all straightforward calculations.
\end{proof}
  
 Of interest in the next section will be an extension of $(1)^\dualmarker$ in a certain special case. Recall that before Lemma \ref{wk=include} we defined for $P\incl\pre\omega\omega$ the \dff{upward Turing closure}
 \[P^{\geq_T}:=\setof g{(\exists f\in P)\;f\leqt g};\]
 we call $P$ ($\dgp\in\Dgs$) \dff{upward Turing closed} iff $P= P^{\geq_T}$ (for some $P\in\dgp$, $P=P^{\geq_T}$). 
 
 By the proofs of Lemma \ref{wk=include} and Proposition \ref{embedings},
 
 \begin{lemma}\label{simpleutc}
 For any upward Turing closed sets $P,Q\incl\pre\omega\omega$,
 \begin{enumerate}
 \item $P\leqs Q\qIff P\supseteq Q$;
 \item $P\join Q \quad=\quad P\cap Q$.\noproof
 \end{enumerate}
 \end{lemma}
 
 \begin{proposition}\label{utcprops}
 For all upward Turing closed $P$ and all $Q, Q_0,Q_1\incl\pre\omega\omega$,
 \begin{enumerate}
 \item $Q\dlatimpl P\equiv_{\ssf s} \setof h{(\forall g\in Q)\;g\oplus h\in P}$;
 \item $\degs(P)$ is meet-irreducible;
 \item $P\dlatimpl (Q_0\meet Q_1)\equiv_{\ssf s}(P\dlatimpl Q_0)\meet (P\dlatimpl Q_1)$.
 \end{enumerate}
 \end{proposition}
 
 \begin{proof}
 For (i), set $I$ to the right-hand set. Then
 \[Q\dlatimpl P:=\setof{(a)^\frown h}{(\forall g\in Q)\;\set a^{g\oplus h}\in P}\leqs I\]
 witnessed by the mapping $h\mapsto (\bar a)^\frown h$ where $\set{\bar a}^{g\oplus h}=g\oplus h$. On the other hand, $(a)^\frown h\in(Q\dlatimpl P)\Implies h\in I$, since $\set a^{g\oplus h}\leqt g\oplus h$, so $\set a^{g\oplus h}\in P\Implies g\oplus h\in P$. Hence $I\leqs(Q\dlatimpl P)$ via the mapping $(a)^\frown h\mapsto h$.
 
 For (ii), suppose that $Q_0\meet Q_1\leqs P$, so some recursive $\Phi:P\to Q_0\meet Q_1$. Then there exist $\sigma$ and $i<2$ such that $\Phi(\sigma)(0)=i$. Set $\Psi(f)(n):=\Phi(\sigma^\frown f)(n+1)$. Then $\Psi:P\to Q_i$ because
 \[f\in P\qImplies\sigma^\frown f\in P\qImplies\Phi(\sigma^\frown f)=(i)^\frown\Psi(f)\in Q_0\meet Q_1.\]
 
 Finally, for (iii), by $(1)^\dualmarker$ above, we need only construct a recursive
 \[\Psi:\bigl(P\dlatimpl(Q_0\meet Q_1)\bigr)\to\bigl((P\dlatimpl Q_0)\meet(P\dlatimpl Q_1)\bigr).\]
 From $(a)^\frown h\in P\dlatimpl(Q_0\meet Q_1)$ we can compute a least pair $(\sigma_{a,h},i_{a,h})$ such that $\set a^{\sigma_{a,h}\oplus h}(0)=i_{a,h}$ and hence an index $b_{a,h}$ such that for any $f$,
 \[\set{b_{a,h}}^{f\oplus h}(n)=\set a^{({\sigma_{a,h}}^\frown f)}(n+1).\] 
 Set $\Psi\bigl((a)^\frown h\bigr):=(i_{a,h},b_{a,h})^\frown h$. Then for all $f\in P$, also ${\sigma_{a,h}}^\frown f\in P$, so $\set{b_{a,h}}^{f\oplus h}\in Q_{i_{a,h}}$. Hence $(b_{a,h})^\frown h\in P\dlatimpl Q_{i_{a,h}}$ and thus $\Psi\bigl((a)^\frown h\bigr)\in(P\dlatimpl Q_0)\meet (P\dlatimpl Q_1)$.
 \end{proof}
 
Finally, we review some standard algebraic notions as they apply to implicative lattices. 

\begin{definition}
For any lattice $\mathstr{L}$ and $F\incl L$, $F$ is a \dff{filter on} $\mathstr L$ iff $\emptyset\not=F\not=L$, and for all $a,b\in L$,
\begin{enumerate}
\item $a\in F\text{ and }a\leq b\qImplies b\in F$;
\item $a,b\in F\qImplies a\meet b\in F$.
\end{enumerate}
F is a \dff{prime filter on} $\mathstr L$ iff additionally
\begin{enumerate}
\item[(iii)] $a\join b\in F\qImplies a\in F\text{ or }b\in F$.
\end{enumerate}
\end{definition}

The easiest examples of filters are the \dff{principal filters}: for $e\not=\latz$,
\[H_e:=\setof{a\in L}{e\leq a}.\]
Easily, if $\mathstr L$ is distributive, then $H_e$ is prime iff $e$ is join-irreducible. More generally, for any $A\incl L$, if $A$ has the \dff{finite intersection property} ({\ssf FIP}):
\[\text{for all finite }B\incl A,\quad\bigmeet B\not=\latz,\]
then
\[H_A:=\setof{a\in L}{(\exists\text{ finite }B\incl A)\;\bigmeet B\leq a}\]
is called the filter \dff{generated by} $A$.
 
 \begin{lemma}
 For any distributive lattice $\mathstr L$ bounded below, any filter $F$ on $\mathstr L$ and any $c,d\in L$ such that $c\join d\in F$ but both $c,d\notin F$,
 \begin{enumerate}
 \item both $F\cup\set c$ and $F\cup\set d$ have the {\ssf FIP};
 \smallskip
 \item $H_{F\cup\set c}\cap H_{F\cup\set d}=F$.
 \end{enumerate}
 \end{lemma}
 
 \begin{proof} For (i), suppose that $F\cup\set c$ does not have the {\ssf FIP}; since $F$ is closed under meet, there exists $a\in F$ such that $a\meet c=\latz$. But then
 \[d\geq a\meet d=(a\meet c)\join(a\meet d)=a\meet(c\join d)\in F,\]
 so $d\in F$ contrary to hypothesis. For (ii), for any $e\in H_{F\cup\set c}\cap H_{F\cup\set d}$ there exist $a,b\in F$ such that
 \[a\meet c\leq e\quad\text{and}\quad b\meet d\leq e.\]
 Then also
  $a\meet b\meet c\leq e\quad\text{and}\quad a\meet b\meet d\leq e$,
  so
\[e\geq(a\meet b\meet c)\join (a\meet b\meet d)=(a\meet b)\meet(c\join d)\in F,\]
and thus $e\in F$.
\end{proof}

A filter is called \dff{maximal in} a class $\cal G$ of filters iff $F\in{\cal G}$ but $\cal G$ contains no proper extension of $F$ --- that is there is no filter $G$ such that $F\subsetneq G\in{\cal G}$. If ${\cal G}\not=\emptyset$ is closed under unions of chains, then by Zorn's Lemma (for the infinite case) $\cal G$ has a maximal element.

\begin{corollary}\label{primesseparate}
For any distributive lattice $\mathstr L$ and any filter $F$ on $\mathstr L$,
\begin{enumerate}
\item $F$ is maximal in the class of all filters on $\mathstr L\qImplies F$ is prime;
\item for any $a,b\in L$, if $(\forall e\in F)\;e\meet a\not\leq b$, then there exists a prime filter $G\supseteq F$ such that $a\in G$ but $b\notin G$.
\end{enumerate}
\end{corollary}

\begin{proof}
For (i), if $F$ is maximal but not prime, then for some $c,d\notin F$, $c\join d\in F$. But then by (i) of the lemma, $H_{F\cup\set c}$ would be a proper extension of $F$, a contradiction.

For (ii), given $a$ and $b$ as described, set
\[{\cal G}:=\setof G{G\supseteq F\text{ is a filter such that }a\in G\text{ and }b\notin G}.\]
${\cal G}\not=\emptyset$ since $H_{F\cup\set a}\in{\cal G}$. Easily $\cal G$ is closed under unions of chains, so has a maximal element $\overline G$. Suppose towards a contradiction that $\overline G$ is not prime, so for some $c,d\notin \overline G$, $c\join d\in \overline G$. Then by (ii) of the lemma, $b$ does not belong to both $H_{\overline G\cup\set c}$ and $H_{\overline G\cup\set d}$, say $b\notin H_{\overline G\cup\set c}$. Then $H_{\overline G\cup\set c}$ is a proper extension of $\overline G$ in $\cal G$, contrary to the choice of $\overline G$.
\end{proof}

\begin{corollary}\label{presimpl}
For any implicative lattice $\mathstr L$, any filter $F$ on $\mathstr L$ and any $a,b\in L$,
\[(a\latimpl b)\in F\qIff (\forall \text{ prime filters }G\supseteq F)\; a\in G\Implies b\in G.\]
\end{corollary}

\begin{proof}
Since $a\meet(a\latimpl b)\leq b$, if both $a$ and $(a\latimpl b)$ are in $G$, then so is $b$. Conversely, if $(a\latimpl b)\notin F$, then for each $e\in F$, $e\not\leq(a\latimpl b)$ so $e\meet a\not\leq b$. Thus the result follows by (ii) of the preceding corollary.
\end{proof}

\begin{remark}
All of these phenomena may also be described in terms of ideals: $I\incl L$ is an \dff{ideal on} $\mathstr L$ iff $\emptyset\not= I\not=L$ and for all $a,b\in L$,
\begin{enumerate}
\item $b\in I\text{ and }a\leq b\qImplies a\in I$;
\item $a,b\in I\qImplies a\join b\in I$;
\end{enumerate}
and a \dff{prime ideal} iff additionally
\begin{enumerate}
\item[(iii)] $a\meet b\in I\qImplies a\in I\text{ or }b\in I$.
\end{enumerate}
A filter (ideal) on ${\mathstr L}^\dualmarker$ is an ideal (filter) on $\mathstr L$ and $F$ is a prime filter (ideal) on $\mathstr L$ iff $L\setminus F$ is a prime ideal (filter) on $\mathstr L$.
\end{remark}

\begin{remark}
As in other branches of algebra, ideals and filters on lattices lead to factor structures. If $F$ is a filter on a lattice $\mathstr L$, then
\[a\leq_F b\dIff(\exists e\in F)\;a\meet e\leq b\]
is a reflexive and transitive relation on $L$ so
\[a\sim_F b\dIff a\leq_Fb\text{ and }b\leq_Fa\]
is an equivalence relation and $\leq_F$ induces a partial ordering on the set $L/F$ of equivalence classes. Defining $\meet_F$ and $\join_F$ in the obvious way leads to a lattice ${\mathstr L}/F$. In particular, ${\mathstr L}/H_{e}$ is isomorphic to the initial segment ${\mathstr L}[\latz,e]$. Note that if $\mathstr L$ is implicative, then $a\leq_Fb\Iff(a\latimpl b)\in F$.
\end{remark}
\end{section}

\begin{section}{Open algebras and completeness}
To prepare for the proof of Theorem K, but also for independent interest, we develop in this section some new examples of (dual-)implicative lattices and refinements of the \IPC- and \WEM-completeness theorems of Section 14. Recall that we call a partial ordering \dff{bounded (above) (below)} iff it has a greatest or least element or both.

\begin{definition}
For any partial ordering ${\mathstr P}=(P,\,\leq)$,
\begin{enumerate}
\item[\textup{(i)}]${\cal O}({\mathstr P}):=\setof{A\incl P}{(\forall a,b\in P)\;a\in A\hbox{ and }a\leq b\Implies b\in A}$;
\item[\textup{(ii)}]${\mathstr O}({\mathstr P}):=({\cal O}({\mathstr P}),\,\emptyset,\,P,\,\cap,\,\cup,\,\incl)$.
\end{enumerate}
\end{definition}

Members of ${\cal O}({\mathstr P})$ are described as the open sets in the order topology. Hence it is just a variation on the example preceding Proposition \ref{DgwIsImpl} to observe that

\begin{lemma} \label{irreducible} For every partial ordering $\mathstr P$, 
\begin{enumerate}
\item[\textup{(i)}]${\mathstr O}({\mathstr P})$ is an implicative lattice;
\item[\textup{(ii)}]if $\mathstr P$ is bounded below, then ${\mathstr O}({\mathstr P})$ is \lato-irreducible;
\item[\textup{(iii)}]if $\mathstr P$ is bounded above, then ${\mathstr O}({\mathstr P})$ is \latz-irreducible.
\end{enumerate}
\end{lemma}

\begin{proof}
Part (i) is straightforward using the implication operator
\[A\latimpl B:=\bigcup\setof{C\in{\cal O}({\mathstr P})}{A\cap C\incl B}.\]
If $\mathstr P$ has least element \latz, then for $A\in{\cal O}({\mathstr P})$, $A=\lato_{{\mathstr O}(\mathstr P)}=P$ iff $\latz\in A$, so if $A=B\cup C$, then one of $B$ or $C$ contains \latz. If $\mathstr P$ has greatest element \lato, then $A=\latz_{{\mathstr O}(\mathstr P)}=\emptyset$ iff $\lato\notin A$, so if $A=B\cap C$, then one of $B$ or $C$ fails to contains \lato. 
\end{proof}

Our first refinement of the completeness theorems is:

\begin{proposition}\label{IPCWEMpo}
\begin{align*}
\IPC&=\bigcap\bigsetof{\Th\bigl({\mathstr O}({\mathstr P})\bigr)}{{\mathstr P}\text{ is a finite partial ordering bounded below}};\\
\WEM&=\bigcap\bigsetof{\Th\bigl({\mathstr O}({\mathstr P})\bigr)}{{\mathstr P}\text{ is a finite bounded partial ordering}}.
\end{align*}
\end{proposition}

\begin{proof}
We shall show that for any finite \lato-irreducible implicative lattice $\mathstr L$, there exists a finite partial ordering $\mathstr P$ bounded below such that ${\mathstr L}\embeds{\mathstr O}({\mathstr P})$, and if $\mathstr L$ is also \latz-irreducible, then we may choose $\mathstr P$ bounded. Then by Lemma \ref{embedtheory}, $\Th\bigl({\mathstr O}({\mathstr P})\bigr)\incl\Th({\mathstr L})$, so by the completeness theorems,
\begin{align*}
\bigcap&\;\bigsetof{\Th\bigl({\mathstr O}({\mathstr P})\bigr)}{{\mathstr P}\text{ is a finite partial ordering bounded below}}\\
&\quad\incl\bigcap\bigsetof{\Th({\mathstr L})}{{\mathstr L}\hbox{ is a 
finite \lato-irreducible implicative lattice}}\\
\noalign{\smallskip}
&\quad=\IPC,
\end{align*}
and
\begin{align*}
\bigcap&\;\bigsetof{\Th\bigl({\mathstr O}({\mathstr P})\bigr)}{{\mathstr P}\text{ is a finite bounded partial ordering}}\\
&\;\incl\bigcap\bigsetof{\Th({\mathstr L})}{{\mathstr L}\hbox{ is a 
finite \latz- and \lato-irreducible implicative lattice}}\\
\noalign{\smallskip}
&\;=\WEM.
\end{align*}
The converse inclusions are immediate from the lemma.

Fix such an $\mathstr L$ and set
\begin{align*}
P&:=\setof F{F\text{ is a prime filter on }{\mathstr L}};\\
{\mathstr P}&:=(P,\;\incl).
\end{align*}
$\mathstr P$ is bounded below by the unit filter $\set{\lato}$, which is prime because \lato\ is join-irreducible, and if $\mathstr L$ is \latz-irreducible, then also $\mathstr P$ has a greatest element $L\setminus\set\latz$. 
Define $\eta:{\mathstr L}\to{\cal O}({\mathstr P})$ by
\[\eta(a):=\setof{F\in P}{a\in F}.\]
$\eta$ is injective and respects $\leq$ by Corollary \ref{primesseparate} (ii). It is straightforward to check respect of the lattice operations:
\begin{align*}
\eta(a\meet b)&=\setof{F\in P}{a\meet b\in F}\\
&=\setof {F\in P}{a\in F\text{ and }b\in F}=\eta(a)\cap\eta(b),\\
\eta(a\join b)&=\setof{F\in P}{a\join b\in F}\\
&=\setof {F\in P}{a\in F\text{ or }b\in F}=\eta(a)\cup\eta(b),
\end{align*}
and that $\eta(\latz)=\emptyset$ and $\eta(\lato)=P$.
Finally, by Corollary \ref{presimpl}, for any $F\in P$,
\begin{endproofeqnarray*}
F\in\eta(a\latimpl b)&\Iff&(\forall\text{ prime }G\supseteq F)\;G\in\eta(a)\Implies G\in\eta (b)\\
&\Iff&\eta(a)\cap\setof{G\in P}{F\incl G}\incl\eta(b)\\
&\Iff& F\in\bigl(\eta(a)\latimpl_{{\mathstr O}({\mathstr P})}\eta(b)\bigr).
\end{endproofeqnarray*}
\end{proof}

\begin{remark}
The existence of embeddings ${\mathstr L}\embeds{\mathstr O}({\mathstr P})$ in the preceding proof can be viewed as a representation theorem for implicative lattices. It is quite parallel to the well-known Stone Representation Theorem for Boolean algebras, both in statement and proof.

The proposition can also be viewed as an alternative formulation of the \dff{Kripke semantics} for intuitionistic propositional logic. Given a partial ordering ${\mathstr P}=(P,\;\leq)$ and a valuation $v:{\ssf PS}\to{\cal O}({\mathstr P})$, define
\begin{equation*}
a\forces\phi\qIff a\in v(\phi).\tag {*}
\end{equation*}
This \dff{forcing relation} on $\mathstr P$ easily satisfies the conditions $a\not\forces\bot$,
\begin{align*}
a\forces\phi\text{ and }a\leq b&\;\qImplies b\forces\phi\\
a\forces\phi\land\psi&\qIff a\forces\phi\text{ and }a\forces\psi\\
a\forces\phi\lor\psi&\qIff a\forces\phi\text{ or }a\forces\psi\\
a\forces\phi\implies\psi&\qIff(\forall b\geq a) \;b\forces\phi\Implies b\forces\psi.
\end{align*}
Any ${\mathstr M}=(P,\;\leq,\;\forces)$ with these properties is called a \dff{Kripke model}. Conversely, given a Kripke model, the function $v$ defined by (*) is a valuation.

A sentence $\phi$ is \dff{true in} $\mathstr M$  --- in symbols, $\modelof{\mathstr M}\phi$ --- iff $(\forall a\in P)\;a\forces \phi$ --- that is, $v(\phi)=P=\lato_{{\mathstr O}({\mathstr P})}$. Hence
\[\Th\bigl({\mathstr O}({\mathstr P})\bigr)=\Bigsetof\phi{\text{for all forcing relations $\forces$ on }{\mathstr P},\;\modelof{(P,\;\leq,\;\forces)}\phi}\]
so by the proposition,
\[\IPC=\Bigsetof\phi{\text{for all Kripke models }{{\mathstr M},\;\modelof{\mathstr M}\phi}},\]
which is one version of the Kripke Completeness Theorem. Of course, we have also the stronger version which restricts to Kripke models based on finite partial orders bounded below.
\end{remark}

Looking at these algebras in a slightly different way, leads to another useful algebra.

\begin{definition}
For any partial ordering ${\mathstr P}=(P,\,\leq)$,
\begin{enumerate}
\item[\textup{(i)}] for any $A\incl P$, $A^*:=\setof{b\in P}{(\exists a\in A)\;a\leq b}$;
\item[\textup{(ii)}] ${\cal O}^\omega({\mathstr P}):=\setof{A^*}{A\incl P\text{ and $A$ is finite}}$.
\end{enumerate}
\end{definition}

Obviously ${\cal O}({\mathstr P})=\setof {A^*}{A\incl P}$, so ${\cal O}^\omega({\mathstr P})\incl{\cal O}({\mathstr P})$ with equality for finite $\mathstr P$. In general, ${\cal O}^\omega({\mathstr P})$ is not naturally the domain of a lattice because it may fail to be closed under intersection. But this problem vanishes under simple natural conditions. Note below that although the notion of dual-implicativity was formally defined only for lattices it applies also to upper semi-lattices.

\begin{proposition}\label{dualimplImpl}
For any upper semi-lattice ${\mathstr P}=(P,\;\leq\;{\bf 0},\;\join)$ that is bounded below,
\begin{enumerate}
\item ${\mathstr O}^\omega({\mathstr P}):=({\cal O}^\omega({\mathstr P}),\;\emptyset,\;P,\;\cap,\;\cup,\;\incl)$ is a bounded lattice;
\item if $\mathstr P$ is dual-implicative, then ${\mathstr O}^\omega({\mathstr P})$ is implicative.
\end{enumerate}
\end{proposition}

\begin{proof}
Fix $\mathstr P$ as described. Clearly ${\cal O}^\omega({\mathstr P})$ is closed under union and contains the least element $\emptyset$. The greatest element $P=\set{{\bf 0}}^*\in{\cal O}^\omega({\mathstr P})$, and for finite $A,B\incl P$,
\[A^*\cap B^*=\setof{a\join b}{a\in A\text{ and }b\in B}^*\in {\cal O}^\omega({\mathstr P}).\]
Thus ${\mathstr O}^\omega({\mathstr P})$ is a lattice. Suppose that $\mathstr P$ is dual-implicative via $\dlatimpl$. Then 
\[A^*\latimpl B^*:=\bigcap_{a\in A}\setof{a\dlatimpl b}{b\in B}^*\]
is an implication operator for ${\mathstr O}^\omega({\mathstr P})$ because for all finite $A,B,C\incl P$,
\begin{endproofeqnarray*}
C^*\incl A^*\latimpl B^*&\Ifff&(\forall c\in C)(\forall a\in A)(\exists b\in B)\;(a\dlatimpl b)\leq c\\
&\Ifff&(\forall c\in C)(\forall a\in A)(\exists b\in B)\;b \leq a\join c\\
&\Ifff&(\forall c\in C)(\forall a\in A)\;a\join c\in B^*\\
&\Ifff&A^*\cap C^*\incl B^*.
\end{endproofeqnarray*}
\end{proof}

Next we look at some particular choices for $\mathstr P$; here, as usual, $n=\set{0,1,\ldots,n-1}$. Much of the material of the rest of this section is taken from Maksimova et. al. \cite{maksetal}.

\begin{definition}
For all $n>0$,
\begin{enumerate}
\item ${\mathstr P}_n:=\bigl(\power(n),\,\supseteq\bigr)$;\qquad ${\mathstr P}_n^-:=\bigl(\power(n)\setminus\set\emptyset,\,\supseteq\bigr)$;
\medskip
\item ${\mathstr O}_n:={\mathstr O}({\mathstr P}_n)$;\qquad ${\mathstr O}_n^-:={\mathstr O}({\mathstr P}_n^-)$;
\medskip
\item ${\mathstr P}_\omega:=\bigl(\setof{a\incl\omega}{a\text{ is finite or cofinite}},\;\supseteq\bigr)$;
\medskip
\item ${\mathstr O}_\omega:={\mathstr O}^\omega({\mathstr P}_\omega)$.
\end{enumerate}
\end{definition}

Of course, in (iv), the hypotheses of Proposition \ref{dualimplImpl} are satisfied because ${\mathstr P}_\omega$ is a Boolean algebra.
Immediately from Proposition \ref{irreducible},

\begin{corollary}
For all $n>0$,
\begin{enumerate}
\item ${\mathstr O}^-_n$ is a \lato-irreducible implicative lattice;
\item ${\mathstr O}_n$ and ${\mathstr O}_\omega$ are  \latz-\ and \lato-irreducible implicative lattices.\noproof
\end{enumerate}
\end{corollary}

\def\LM{{\ssf LM}}

By Propositions \ref{IPCinclThL} and \ref{WEMinclThL}, we have
\[\IPC\incl\bigcap\setof{\Th({\mathstr O}_n^-)}{n>0}\qand \WEM\incl\bigcap\setof{\Th({\mathstr O}_n)}{n>0}.\]
We shall see that the second inclusion is in fact an equality, but the first is not and leads to a new logic, which is denoted \LM\ for the Logic of Medvedev, who first considered it (with a different definition -- see \cite{MedLogic} and Section 6 of Gabbay \cite{Gabbay}).

\begin{definition}
$\LM:=\bigcap\setof{\Th({\mathstr O}_n^-)}{n>0}$.
\end{definition}

\reldef0\conecovers[\buildrel\bigtriangleup\over\longrightarrow]

\begin{definition}
For any partial orderings ${\mathstr P}=(P,\,\leq_P)$ and ${\mathstr Q}=(Q,\,\leq_Q)$, ${\mathstr P}$ \dff{cone-covers} ${\mathstr Q}$ --- in symbols, ${\mathstr P}\conecovers{\mathstr Q}$ --- iff there exists a surjective function $f:P\to Q$ such that for all $a\in P$, $f$ maps \setof{b\in P}{a\leq_Pb} onto \setof{d\in Q}{f(a)\leq_Qd}. We write $f:{\mathstr P} \conecovers{\mathstr Q}$.
\end{definition}

\begin{proposition}\label{conecovertheories}
For any partial orderings ${\mathstr P}=(P,\,\leq_P)$ and ${\mathstr Q}=(Q,\,\leq_Q)$
\[{\mathstr P}\conecovers{\mathstr Q}\qImplies\Th\bigl({\mathstr O}({\mathstr P})\bigr)\incl\Th\bigl({\mathstr O}({\mathstr Q})\bigr).\]
\end{proposition}

\begin{proof}
Fix a function $f$ witnessing ${\mathstr P}\conecovers{\mathstr Q}$ and for any ${\mathstr O}({\mathstr Q})$ valuation $v$, set
\[v^f(\phi):=f^{-1}(v(\phi)).\]
The properties of $f$ guarantee that $f$ is order-preserving, so easily $v^f:{\ssf PS}\to{\cal O}({\mathstr P})$. We show first that in fact $v^f$ is an ${\mathstr O}({\mathstr P})$-valuation. The conditions
\[v^f(\phi\land\psi)=v^f(\phi)\meet v^f(\psi)\qand v^f(\phi\lor\psi)=v^f(\phi)\join v^f(\psi)\]
follow from the elementary properties of inverse images: for $C,D\in{\cal O}({\mathstr Q})$,
\[f^{-1}(C\meet D)=f^{-1}(C\cap D)=f^{-1}(C)\cap f^{-1}(D)=f^{-1}(C)\meet f^{-1}(D)\]
and similarly for $\join$. Since also obviously $f^{-1}(\latz_{{\mathstr O}({\mathstr Q})})=\latz_{{\mathstr O}({\mathstr P})}$, the remaining two conditions follow once we establish that for all  $C,D\in{\cal O}({\mathstr Q})$,
\[f^{-1}(C\latimpl D)=f^{-1}(C)\latimpl f^{-1}(D),\]
where the lattice implications are respectively those of ${\mathstr O}({\mathstr Q})$ and ${\mathstr O}({\mathstr P})$. For this, we need to verify that for any $X\in{\cal O}({\mathstr P})$,
\[f^{-1}(C)\cap X\incl f^{-1}(D)\qIff X\incl f^{-1}(C\latimpl D).\]
For (\Larrow) we have by the properties of $\latimpl$ in ${\mathstr O}({\mathstr Q})$,
\[f^{-1}(C)\cap f^{-1}(C\latimpl D)=f^{-1}\bigl(C\cap(C\latimpl D)\bigr)\incl f^{-1}(D).\]
Towards (\Rarrow), suppose that $f^{-1}(C)\cap X\incl f^{-1}(D)$ and $a\in X$. Then for all $b\in P$, $a\leq_Pb\Implies b\in X$, so
\[a\leq_Pb\qand f(b)\in C\qImplies f(b)\in D.\]
By cone-covering, for each $d\geq_Qf(a)$ there exists $b\geq_Pa$ with $d=f(b)$, so we have
\[f(a)\leq_Qd\qand d\in C\qImplies d\in D.\]
Hence 
\[f(a)\in Y:=\setof{d\in Q}{f(a)\leq_Qd}\in{\cal O}({\mathstr Q}) \qand C\cap Y\incl D,\]
so $Y\incl C\latimpl D$ and in particular $f(a)\in C\latimpl D$ and thus $a\in f^{-1}(C\latimpl D)$.

Finally, we have the desired conclusion: for any propositional sentence $\phi$,
\begin{endproofeqnarray*}
\modelof{{\mathstr O}({\mathstr P})}\phi
&\qImplies&\text{for all ${\mathstr O}({\mathstr Q})$-valuations $v$, }v^f(\phi)=\lato_{{\mathstr O}({\mathstr P})}\\
&\qImplies&\text{for all ${\mathstr O}({\mathstr Q})$-valuations $v$, }v(\phi)=\lato_{{\mathstr O}({\mathstr Q})}\\
&\qImplies& \modelof{{\mathstr O}({\mathstr Q})}\phi.
\end{endproofeqnarray*}
\end{proof}

\begin{proposition}
For any finite bounded partial ordering ${\mathstr Q}$, there exists $m>0$ such that ${\mathstr P}_m\conecovers{\mathstr Q}$.
\end{proposition}

\mdef0\zQ[{\bf 0}_{\mathstr Q}]
\mdef0\oQ[{\bf 1}_{\mathstr Q}]

\begin{proof}
We proceed by induction on the size of $\mathstr Q$. If $\mathstr Q$ has only one element, the conclusion is clear. Let \zQ\ and \oQ\ denote respectively the least and greatest element of $\mathstr Q$ and $e_0,\ldots,e_k$ be the immediate successors of \zQ. Suppose first that $k=0$ and set $R:=Q\setminus\set{\zQ}$, $\leq_R$ is $\leq_Q$ restricted to $R$ and ${\mathstr R}:=(R,\,\leq_R)$. By the induction hypothesis, for some $m>0$ and $g$, $g:{\mathstr P}_m\conecovers{\mathstr R}$. Then we show that $f:{\mathstr P}_{m+1}\conecovers{\mathstr Q}$ with, for $A\incl m+1$,
\begin{equation*}
f(A):=
\cases{
\zQ,			&\text{if $m\in A$ and $g(A\cap m)=e_0$};\cr
\noalign{\smallskip}
g(A\cap m),	&\text{otherwise.}\cr}
\end{equation*}
Clearly $f$ is surjective. The cone-covering condition is, for all $d\in Q$,
\begin{equation*}
f(A)\leq_Q d\qIff(\exists B\incl A)\;f(B)=d.\tag{*}
\end{equation*}
The implication (\Larrow) follows because $B\incl A\Implies B\cap m\incl A\cap m$ and $g$ is order-preserving. For (\Rarrow), if $d=\zQ$, take $B=A$; otherwise $d\in R$. If $f(A)=\zQ$, then $g(A\cap m)=e_0\leq_Rd$ and there exists $B\incl A\cap m$ with $g(B)=d$, so also $f(A)=d$. If $f(A)=g(A\cap m)$, then again the desired $B$ exists by the cone-covering property of $g$.

Now, consider the case $k>0$. For $i<k$, set $R_i:=\setof{c\in Q}{e_i\leq_Qc}\cup\set{\zQ}$, $\leq_{R_i}$ is $\leq_Q$ restricted to $R_i$ and ${\mathstr R}_i:=(R_i,\,\leq_{R_i})$. By the induction hypothesis, for all $i<k$ there exist $m_i$ and $g_i$ with $g_i:{\mathstr P}_{m_i}\conecovers{\mathstr R}_i$. Let $m:=m_0+\cdots+m_k$. We shall show that $f:{\mathstr P}_m\conecovers{\mathstr Q}$ with $f$ defined as follows. For $A\incl m$ and $i<k$, set
\begin{align*}
A^{(i)}&:=\setof{j<m_i}{m_0+\cdots+m_{i-1}+j\in A}.\\
\noalign{\smallskip}
f(A)&:=\cases{g_i\bigl(A^{(i)}\bigr),&if $(\forall j\not=i)\,g_j(A^{(j)})=\zQ$;\cr
\noalign{\smallskip}
\oQ,&if for at least two $i<k$, $g_i\bigl(A^{(i)}\bigr)\not=\zQ$.\cr}
\end{align*}
To establish (*)$(\Larrow)$ suppose that $B\incl A$. If for at least two $i<k$, $g_i\bigl(B^{(i)}\bigr)\not=\zQ$, then $f(B)=\oQ$ so $f(A)\leq_Q f(B)$. Otherwise, for some $i$, $(\forall j\not=i)\,g_j(B^{(j)})=\zQ$. Since for all $j<k$, $B^{(j)}\incl A^{(j)}$, $g_j\bigl(A^{(j)}\bigr)\leq_Q g_j\bigl(B^{(j)}\bigr)$, so also $(\forall j\not=i)\,g_j\bigl(A^{(j)}\bigr)=\zQ$, and 
\[f(A)=g_i\bigl(A^{(i)}\bigr)\leq g_i\bigl(B^{(i)}\bigr)=f(B).\]
Now, towards (*)$(\Rarrow)$, suppose that $f(A)\leq_Qd$. If for at least two $i<k$ $g_i\bigl(A^{(i)}\bigr)\not=\zQ$, then $f(A)=\oQ$ so also $d=\oQ$. Suppose that for at most one $i<k$,  $g_i\bigl(A^{(i)}\bigr)\not=\zQ$. Then either $d=\zQ$, so $f(A)=d$, or for some $i$, $e_i\leq_Qd$, $f(A)=g_i\bigl(A^{(i)}\bigr)\leq_Qd$ and $(\forall j\not=i)\,g_j\bigl(A^{(j)}\bigr)=\zQ$. Then by the cone-covering property of $g_i$, there exists $C\incl A^{(i)}$ such that $g_i(C)=d$ and $f(B)=d$ for $B$ such that $B^{(i)}=C$ and for $j\not=i$, $B^{(j)}=A^{(j)}$.
\end{proof}

\begin{corollary}\label{WEMOn}
$\WEM=\bigcap\bigsetof{\Th({\mathstr O}_n)}{n>0}.$
\end{corollary}

\begin{proof}
Immediate from the preceding proposition and Propositions \ref{IPCWEMpo} and  \ref{conecovertheories}.
\end{proof}

\begin{lemma}
For each $n>0$, ${\mathstr P}_{n+1}^-\conecovers{\mathstr P}_n$.
\end{lemma}

\begin{proof}
For nonempty $X\incl n+1$, set $\mu(X):={}$the smallest element of $X$ and
\[f(X):=\bigsetof i{i+1\in X\setminus\set{\mu(X)}}.\]
We show that $f:{\mathstr P}_{n+1}^-\conecovers{\mathstr P}_n$. $f$ is surjective because, for any $Z\incl n$, $f(Y)=Z$ for $Y=\setof{i+1}{i\in Z}\cup\set0$. To see that $f$ is cone-covering, for any $W\incl f(X)$, set $Y=\setof{i+1}{i\in W}\cup\set{\mu(X)}$. Then $Y\incl X$, since
\[i+1\in Y\setminus\set{\mu(X)}\Implies i\in W\Implies i\in f(X)\Implies i+1\in X,\]
so $\mu(Y)=\mu(X)$ and thus $f(Y)=W$.
\end{proof}

\begin{corollary}
$\LM\incl\WEM$.
\end{corollary}

\begin{proof}
By the lemma and Proposition \ref{conecovertheories}, $\Th({\mathstr O}_{n+1}^-)\incl\Th({\mathstr O}_n)$. Hence the conclusion follows from the definition of \LM\ and Proposition \ref{WEMOn}.
\end{proof}

We defined the notion of positive propositional sentence just before Lemma \ref{LIPCimplic}; let {\ssf Pos} denote the set of these.

\begin{proposition}\label{WEMposIPC}
$\WEM\cap{\ssf Pos}=\IPC\cap{\ssf Pos}$; hence also $\LM\cap{\ssf Pos}=\IPC\cap{\ssf Pos}$.
\end{proposition}

\begin{proof}
The inclusion $(\supseteq)$ is clear. Fix $\phi\in\WEM\cap{\ssf Pos}$ and $\mathstr L$ a finite \lato-irreducible implicative lattice. By Lemma \ref{lzo}, $\lz$ is also \latz-irreducible, so by the \WEM-Completeness Theorem, $\modelof{\lz}\phi$. But then, since $\phi$ is positive, it follows from Lemma \ref{zirred} that also  $\modelof{{\mathstr L}}\phi$. Hence by the \IPC-Completeness Theorem, $\phi\in\IPC$. The second clause now follows by the preceding corollary.
\end{proof}

Next we adapt these to yield characterizations of \IPC\ in terms of intervals ${\mathstr L}[d,\lato]$. 

\begin{definition} To each propositional sentence $\phi$ we associate a positive sentence $\phi^+$ as follows. Let $n_\phi$ be the smallest $n$ such that all atomic sentences occurring in $\phi$ are among $\p0,\ldots,\p n$. Define $\psi^\phi$ recursively on subsentences of $\phi$ by:
\begin{align*}
&\p i^\phi:=\p i\quad\text{for each atomic sentence }\p i;\quad
(\psi\land\theta)^\phi:=\psi^\phi\land\theta^\phi;\\
&(\psi\lor\theta)^\phi:=\psi^\phi\lor\theta^\phi;\qquad
(\psi\implies\theta)^\phi:=\psi^\phi\implies\theta^\phi;\\
&(\lnot\psi)^\phi:=\psi^\phi\implies\p0\land\ldots\land\p{n_\phi}\land\p{n_\phi+1}.
\end{align*}
Then $\phi^+:=\phi^\phi$.
\end{definition}

\begin{lemma}
For any propositional sentence $\phi$ and any implicative lattice $\mathstr L$,
\begin{enumerate}
\item $\modelof{{\mathstr L}}\phi^+\qImplies\modelof{{\mathstr L}}\phi$;
\item $\notmodelof{{\mathstr L}}\phi^+\qImplies(\exists d\in L)\notmodelof{{\mathstr L}[d,\lato]}\phi$.
\end{enumerate}
\end{lemma}

\begin{proof}
Fix $\mathstr L$ and $\phi$. For each $\mathstr L$-valuation $v$, let $v_\phi$ be the $\mathstr L$-valuation that agrees with $v$ except that $v_\phi(\p{n_\phi+1})=\latz$. By an easy induction on subsentences $\psi$ of $\phi$, $v_\phi(\psi^\phi)=v(\psi)$, so in particular $v_\phi(\phi^+)=v(\phi)$. Thus, if $\modelof{{\mathstr L}}\phi^+$, then for each $\mathstr L$-valuation $v$, $v_\phi(\phi^+)=\lato$ so also $v(\phi)=\lato$. Hence $\modelof{{\mathstr L}}\phi$.

Towards (ii), suppose that for some $\mathstr L$-valuation $v$, $v(\phi^+)\not=\lato$. Set
\[d:=v(\p0)\meet\cdots\meet v(\p{n_\phi})\meet v(\p{n_\phi+1}).\]
There is a unique $\mathstr L[d,\lato]$-valuation $w$ such that for $i\leq n_\phi+1$, $w(\p i)=v(\p i)$ and for $i>n_\phi+1$, $w(\p i)=\lato$. Then for each subsentence $\psi$ of $\phi$, $w(\psi)=v(\psi^\phi)$, since inductively
\[w(\lnot\psi)=w(\psi)\latimpl\latz_{{\mathstr L}[d,\lato]}=v(\psi^\phi)\latimpl d=v((\lnot\psi)^\phi).\]
In particular, $w(\phi)=v(\phi^+)\not=\lato$, so $\notmodelof{{\mathstr L}[d,\lato]}\phi$.
\end{proof}

\begin{proposition}\label{intersectseg}
For any collection $\cal X$ of implicative lattices, if
\[\IPC\cap{\ssf Pos}=\bigcap_{{\mathstr L}\in{\cal X}}\Th({\mathstr L})\cap{\ssf Pos},\]
then
\[\IPC=\bigcap_{{\mathstr L}\in{\cal X}}\bigcap_{d\in L}\Th\bigl({\mathstr L}[d,\lato]\bigr).\]
\end{proposition}

\begin{proof}
Assume the hypothesis. The inclusion $(\incl)$ of the conclusion is immediate. Suppose that a sentence $\phi$ belongs to the right-hand side. Then for each ${\mathstr L}\in{\cal X}$, by (ii) of the preceding lemma, $\modelof{{\mathstr L}}\phi^+$, so since $\phi^+$ is positive, $\phi^+\in\IPC$. Hence, for {\it every} finite implicative lattice $\mathstr L$, $\modelof{{\mathstr L}}\phi^+$, so by (i) of the lemma, $\modelof{{\mathstr L}}\phi$. Thus by the \IPC-Completeness Theorem, $\phi\in\IPC$.
\end{proof}

\begin{corollary}\label{OnCompleteness}
\[\bigcap_{n>0}\bigcap_{D\in{\cal O}({\mathstr P}_n)}\Th\bigl({\mathstr O}_n[D,\lato]\bigr)
=\IPC
 =\bigcap_{n>0}\bigcap_{D\in{\cal O}({\mathstr P}_n^-)}\Th\bigl({\mathstr O}_n^-[D,\lato]\bigr).\]
 \end{corollary}
 
 \begin{proof}
 Immediate from Corollary \ref{WEMOn} and Propositions \ref{WEMposIPC} and \ref{intersectseg}.
 \end{proof}
 
 \begin{corollary}
 \begin{align*}
 \bigcap_{\dgr\in\Dgs}\Th^\dualmarker\bigl(\Dgs[\latz,\dgr]\bigr)&=\IPC\\
 \bigcap_{\dgr\in\Dgw}\Th^\dualmarker\bigl(\Dgw[\latz,\dgr]\bigr)&=\IPC\\
 \bigcap_{\dgr\in\Dgw}\Th\bigl(\Dgw[\dgr,\lato]\bigr)&=\IPC.
 \end{align*}
 \end{corollary}
 
 \begin{proof}
 This follows by Theorem J and Propositions \ref{WEMposIPC} and \ref{intersectseg}. For example, since
 \[\Th\bigl(\Dgs^\dualmarker\bigr)\cap{\ssf Pos}=\WEM\cap{\ssf Pos}=\IPC\cap{\ssf Pos},\]
 we have
 \begin{endproofeqnarray*}
 \IPC&=\bigcap_{\dgr\in\Dgs}\Th\bigl(\Dgs^\dualmarker[\dgr,\lato]\bigr)\\
&=\bigcap_{\dgr\in\Dgs}\Th\bigl(\Dgs[\latz,\dgr]^\dualmarker\bigr)\\
&=\bigcap_{\dgr\in\Dgs}\Thd\bigl(\Dgs[\latz,\dgr]\bigr).
\end{endproofeqnarray*}
\end{proof}
\end{section}

\begin{section}{Proof of Theorem K}
Improving on the preceding corollary, we shall construct a single strong degree \dgr\ such that $\Thd\bigl(\Dgs[\latz,\dgr]\bigr)=\IPC$. We follow generally the presentation of Skvortsova \cite{Sk}. Fix an enumeration $\functionof{(n_k,D_k)}{k\in\omega}$ of all pairs $(n,D)$ with $D\in{\cal O}_n$ so that by Corollary \ref{OnCompleteness},
\begin{equation*}
\bigcap_{k\in\omega}\Th\bigl({\mathstr O}_{n_k}[D_k,\lato]\bigr)=\IPC.\tag{1}
\end{equation*}
We shall construct strong degrees $\dgp_k$, $\dgq_k$ and $\dgr$ ($k\in\omega$) such that
\begin{equation*}
{\mathstr O}_{n_k}[D_k,\lato]^\dualmarker\dembeds\Dgs[\dgp_k,\dgq_k],\tag{2}
\end{equation*}
and
\begin{equation*}
\dgp_k\join\dgr=\dgq_k.\tag{3}
\end{equation*}
Then \[{\mathstr O}_{n_k}[D_k,\lato]\embeds\Dgs[\dgp_k,\dgq_k]^\dualmarker,\]
so by  Corollary \ref{surjectiveTh}(ii) and Lemma \ref{embedtheory},
\[\Thd\bigl(\Dgs[\latz,\dgr]\bigr)\incl\bigcap_{k\in\omega}\Thd\bigl(\Dgs[\dgp_k,\dgq_k]\bigr)
\incl\bigcap_{k\in\omega}\Th\bigl({\mathstr O}_{n_k}[D_k,\lato]\bigr)=\IPC.\]
Towards (2), note first that for each $n$ and each $D\in{\cal O}_n$,
\[{\mathstr O}_n[D,\lato]={\mathstr O}_\omega[D,\power(n)].\]
Hence, if we show that ${\mathstr O}^\dualmarker_\omega\dembeds\Dgs$, then there are $\dgp,\dgq\in\Dgs$ such that
\[{\mathstr O}_n[D,\lato]^\dualmarker={\mathstr O}_\omega^\dualmarker[\power(n),D]\dembeds\Dgs[\dgp,\dgq].\]
We then achieve (3) by careful choice of \dgp\ and \dgq\ for different pairs $(n,D)$.

The basis of the construction is the following classical result.

\begin{proposition}[Lachlan and Lebeuf {\cite{LaLe}}]
For any countable upper semi-lattice
${\mathstr P}=(P,\leq,\latz,\join)$ that is bounded below, there exists an embedding of $\mathstr P$ into the Turing degrees $\Dgt$ as an initial segment --- that is, a function $\xi:P\to\Dgt$ such that for $a,b\in P$,
\begin{enumerate}
\item $a\leq b\Ifff\xi(a)\leqt\xi(b)$;
\smallskip
\item $\xi(\latz)=\latz_T$;
\smallskip
\item $\xi(a\join b)=\xi(a)\join\xi(b)$;
\smallskip
\item ${\ssf Im}(\xi)$ is closed downward in \Dgt.\noproof
\end{enumerate}
\end{proposition}

\begin{proposition}\label{PomegaEmbeds}
For any countable dual-implicative upper semi-lattice ${\mathstr P}=(P,\leq,\latz,\join,\dlatimpl)$, there exists a dual embedding $\overline\eta:{\mathstr O}^\omega({\mathstr P})^\dualmarker\dembeds\Dgs$ such that for each $D\in{\cal O}^\omega({\mathstr P})$, $\overline\eta(D)$ is the meet of finitely many upward Turing closed degrees.
\end{proposition}

\begin{proof}
By Proposition \ref{dualimplImpl}, ${\mathstr O}^\omega({\mathstr P})^\dualmarker$ is dual-implicative. Fix an upper semi-lattice embedding $\xi:{\mathstr P}\to\Dgt$ as in the preceding proposition. We first transform this into an upper semi-lattice embedding $\eta:{\mathstr P}\to\Dgs$:
\[\eta(a):=\degs(S_a)\text{\quad where}\]
\[S_a:=\setof{f\in\pre\omega\omega}{\xi(a)\leqt\degt(f)\text{ or }\degt(f)\notin{\ssf Im}(\xi)}.\]
Each $S_a$ is upward Turing closed. It follows by Lemma \ref{simpleutc} that $\eta$ respects $\leq$:
\[a\leq b\Ifff\xi(a)\leqt\xi(b)\Ifff S_b\incl S_a\Ifff S_a\leqs S_b.\]
$\eta(\latz)=\degs(S_\latz)=\degs(\pre\omega\omega)=\latz_{\ssf s}$, and again by Lemma \ref{simpleutc}, $\eta$ respects $\join$:
\begin{align*}
S_{a\join b}&=\setof f{\xi(a)\join\xi(b)\leqt\degt(f)\text{ or }\degt(f)\notin{\ssf Im}(\xi)}\\
&=S_a\cap S_b\equiv_{\ssf s}S_a\join S_b.
\end{align*}
Next we show that $\eta$ respects $\dlatimpl$. For $\degt(f)\notin{\ssf Im}(\xi)$, $f\in S_{a\dlatimpl b}$ by definition, and by Proposition \ref{utcprops}, also 
\[f\in S_a\dlatimpl S_b=\setof h{(\forall g\in S_a)\;g\oplus h\in S_b},\] 
since 
\[\degt(f)\notin{\ssf Im}(\xi)\Impliess\degt(g\oplus f)\notin{\ssf Im}(\xi)\Impliess g\oplus f\in S_b.\]
For $\degt(f)\in {\ssf Im}(\xi)$, fix $c\in P$ such that $\xi(c)=\degt(f)$. Then
\begin{align*}
f\in S_{a\dlatimpl b}&\Ifff\xi(a\dlatimpl b)\leqt\degt(f)\\
&\Ifff(a\dlatimpl b)\leq c\\
&\Ifff b\leq a\join c\\
&\Ifff \xi(b)\leqt\xi(a)\join\degt(f)\\
&\Ifff \forall g[\xi(a)\leqt\degt(g)\Impliess\xi(b)\leqt\degt(f\oplus g)]\\
&\Ifff\bigl(\forall g\in S_a\cap{\ssf Im}(\xi)\bigr)\;\;\xi(b)\leqt\degt(f\oplus g)\\
&\Ifff f\in S_a\dlatimpl S_b,
\end{align*}
with the last step by Proposition \ref{utcprops} (i).

Now, extend $\eta$ to $\overline\eta:{\mathstr O}^\omega(\mathstr P)^\dualmarker\to\Dgs$ by setting for finite $A\incl P$,
\[\overline\eta(A^*):=\bigmeet_{a\in A}\eta(a).\]
It remains to show that $\overline\eta$ is well-defined and is a dual-implicative lattice embedding --- that is
\begin{enumerate}
\item $A^*\supseteq B^*\Ifff\overline\eta(A^*)\leqs\overline\eta(B^*)$;
\smallskip
\item $\overline\eta(\emptyset)=\infty_{\ssf s};\quad\overline\eta(P)=\latz_{\ssf s}$;
\smallskip
\item $\overline\eta(A^*\cap B^*)=\overline\eta(A^*)\join\overline\eta(B^*)$;
\smallskip
\item $\overline\eta(A^*\cup B^*)=\overline\eta(A^*)\meet\overline\eta(B^*)$;
\smallskip
\item $\overline\eta(A^*\latimpl B^*)=\overline\eta(A^*)\dlatimpl\overline\eta(B^*)$.
\end{enumerate}
For (i), which also implies that $\overline\eta$ is well-defined, we have
\begin{align*}
A^*\supseteq B^*&\Ifff(\forall b\in B)\;(\exists a\in A)\;a\leq b\\
&\Ifff(\forall b\in B)\;(\exists a\in A)\;\eta(a)\leqt\eta(b)\\
&\Ifff(\forall b\in B)\;\overline\eta(A^*)\leqs\eta(b)\\
&\Ifff\overline\eta(A^*)\leqs\overline\eta(B^*),
\end{align*}
where the third equivalence uses the meet irreducibility of $\eta(b)$ from Proposition \ref{utcprops}(ii). Part (ii) is immediate. For (iii), we have
\begin{align*}
\overline\eta(A^*\cap B^*)&=\overline\eta\bigl(\setof{a\join b}{a\in A\text{ and }b\in B}^*\bigr)\\
&=\bigmeet_{a\in A}\bigmeet_{b\in B}\eta(a\join b)\\
&=\bigmeet_{a\in A}\bigmeet_{b\in B}\eta(a)\join \eta(b)\\
&=\overline\eta(A^*)\join\overline\eta(B^*),
\end{align*}
by distributivity. Part (iv) is straightforward:
\[\overline\eta(A^*\cup B^*)=\overline\eta\bigl((A\cup B)^*\bigr)=\bigmeet_{a\in A\cup B}\eta(a)=\overline\eta(A^*)\meet\overline\eta(B^*).\]
Finally for (v) we have
\begin{endproofeqnarray*}
\overline\eta(A^*\latimpl B^*)
&=&\overline\eta\Bigl(\bigcap_{a\in A}\setof{a\dlatimpl b}{b\in B}^*\Bigr)\\
&=&\overline\eta\Bigl(\Bigsetof{\bigjoin_{a\in A}\bigl(a\dlatimpl F(a)\bigr)}{F\in\pre BA}^*\Bigr)\\
&=&\bigmeet_{F\in\pre BA}\eta\Bigl(\bigjoin_{a\in A}\bigl(a\dlatimpl F(a)\bigr)\Bigr)\\
&=&\bigmeet_{F\in\pre BA}\bigjoin_{a\in A}\eta\bigl(a\dlatimpl F(a)\bigr)\\
&=&\bigjoin_{a\in A}\bigmeet_{b\in B}\eta(a)\dlatimpl\eta(b)\\
&=&\bigjoin_{a\in A}\bigl(\eta(a)\dlatimpl\overline\eta(B^*)\bigr)\quad\text{by Proposition \ref{utcprops}(iii)}\\
&=&\overline\eta(A^*)\dlatimpl\overline\eta(B^*)\quad\text{by Lemma \ref{distprops} $(4)^\dualmarker$}.
\end{endproofeqnarray*}
\end{proof}

We proceed now to establishing (2) and (3) to complete the proof of Theorem K. By the preceding proposition we may fix $\overline\eta:{\mathstr O}_\omega^\dualmarker:={\mathstr O}^\omega\bigl({\mathstr P}_\omega\bigr)^\dualmarker\dembeds\Dgs$. For each $k\in\omega$ choose $a_k\incl\omega$ such that $a_k$ has $n_k$-many elements and $k\not=l\Implies a_k\cap a_l=\emptyset$. Clearly, for each $k$ there exists $E_k\in{\cal O}\bigl(\power(a_k)\bigr)$ such that 
\[{\mathstr O}_{n_k}[D_k,\lato]\iso{\mathstr O}_\omega[E_k,\power(a_k)]\]
and
\[{\mathstr O}_{n_k}[D_k,\lato]^\dualmarker
={\mathstr O}_{n_k}^\dualmarker[\lato,D_k]
\iso{\mathstr O}_\omega^\dualmarker[\power(a_k),E_k].\]
Set 
\[\dgp_k:=\overline\eta\bigl(\power(a_k)\bigr)\text{ and }\dgq_k:=\overline\eta(E_k).\]
Then immediately,
\[{\mathstr O}_{n_k}[D_k,\lato]^\dualmarker\dembeds\Dgs[\dgp_k,\dgq_k]\]
as required by (2). Towards (3), note that obviously
\[\dgp_k\leqs\dgq_k\quad\text{so}\quad \dgp_k\join\dgq_k=\dgq_k.\]
For $k\not=\ell$, we have, since $E_\ell\incl a_\ell$,
\[\dgp_k\join\dgq_\ell=\overline\eta\bigl(\power(a_k)\bigr)\join\overline\eta(E_\ell)=\overline\eta\bigl(\power(a_k)\cap E_l\bigr)=\overline\eta(\set\emptyset)\geq_{\ssf s}\overline\eta(E_k)=\dgq_k.\]
Hence, if \Dgs\ were a complete and completely distributive lattice, we could set
\[\dgr:=\bigmeet_{\ell\in\omega}\dgq_\ell\]
and compute
\[\dgp_k\join\dgr=\bigmeet_{\ell\in\omega}(\dgp_k\join\dgq_\ell)=\dgq_k\]
as required by (3). Since it isn't, we need a slightly more cumbersome construction using the ``pseudo meet" of Remark \ref{genmeet}. Choose $P_k\in\dgp_k$. By Proposition \ref{PomegaEmbeds} for each $\ell\in\omega$ we may choose $Q_\ell\in\dgq_\ell$, $m_\ell\in\omega$ and upward Turing closed sets $Q_{\ell,i}$ such that $Q_\ell=\bigmeet_{i<m_\ell}Q_{\ell,i}$. Set
\[R:=\bigmeet_{\ell\in\omega}Q_\ell:=\setof{(\ell)^\frown g}{\ell\in\omega\text{ and }g\in Q_\ell}\quad\text{and}\quad\dgr:=\degs(R).\]
Clearly $\dgp_k\join\dgr\leqs\dgq_k$. For $k\not=\ell$, we have as above
\[\degs(P_k\join Q_\ell)=\overline\eta(\set\emptyset)\geq_{\ssf s}\degs(Q_{k,0}),\]
so since $Q_{k,0}$ is upward Turing closed, by Proposition \ref{simpleutc}, $P_k\join Q_\ell\incl Q_{k,0}$. Hence, if we set
\[\Phi_k\bigl(f\oplus(\ell)^\frown g\bigr):=\cases{g,&if $k=\ell$;\cr
\noalign{\smallskip}
(0)^\frown(f\oplus g),&otherwise;}\]
we have $\Phi_k:P_k\join R\to Q_k$ --- that is, $\dgq_k\leqs\dgp_k\join\dgr$ as required.
\end{section}

\end{document}
\bye